\documentclass[12pt]{article}

\usepackage[T1]{fontenc}
\usepackage[utf8]{inputenc}
\usepackage[margin=1in]{geometry}
\usepackage{lmodern}
\usepackage{microtype}
\usepackage{amsthm}
\usepackage{amssymb}
\usepackage{enumitem}
\usepackage{newtxtext,newtxmath} 
\usepackage{mathtools}
\usepackage{titlesec,setspace}
\usepackage[authoryear,round]{natbib}
\usepackage[bookmarksnumbered,bookmarksopen,colorlinks,linkcolor={black},urlcolor={black},citecolor={black}]{hyperref}

\expandafter
\def \expandafter \normalsize \expandafter{\normalsize \setlength \abovedisplayskip{10pt plus 2pt minus 7pt}}
\expandafter
\def \expandafter \normalsize \expandafter{\normalsize \setlength \abovedisplayshortskip{0pt plus 2pt}}
\expandafter
\def \expandafter \normalsize \expandafter{\normalsize \setlength \belowdisplayskip{10pt plus 2pt minus 7pt}}
\expandafter
\def \expandafter \normalsize \expandafter{\normalsize \setlength \belowdisplayshortskip{5pt plus 2pt minus 3pt}}

\newtheorem{thm}{Theorem}[section]
\newtheorem{cor}{Corollary}[section]
\newtheorem{lemma}{Lemma}[section]
\newtheorem{proposition}{Proposition}[section]

\theoremstyle{definition}
\newtheorem{remark}{Remark}[section]
\newtheorem{assumption}{Assumption}[section]

\newtheorem{example}{Example}[section]

\mathtoolsset{mathic=true}

\newcommand{\ii}{\mathrm{i}}
\newcommand{\RRe}{\mathrm{Re}}

\newcommand{\Erl}{\mathrm{Erl}}

\newcommand{\Exp}{\mathrm{Exp}}

\newcommand{\PP}{\mathbb{P}}

\providecommand{\mathbbm}[1]{\mathbf{1}}
\numberwithin{equation}{section}

\begin{document}
    \title{\Large Random trade timing and power-law tails in realized prices}
    \author{Won-Ki Seo \\[3pt] University of Sydney}
    \date{}

    \maketitle

\begin{abstract}
This paper studies stochastic mechanisms under which light-tailed latent price dynamics yield realized prices with power-law tails. The realized price is modeled as $P_T=e^{X_T}$, where $X$ is a Markov-modulated L\'evy
process and $T$ is the random time of the next trade. We consider two trade-timing environments. In the intertrade-incidence model, trades occur on a discrete grid with type-dependent probabilities. In the intertrade-time
model, the waiting time to the next trade is generalized Erlang, allowing for heterogeneous arrival rates and transaction-completion delays. We show that random trade timing can generate Pareto-type tails, possibly with a
logarithmic correction, in realized prices even when the latent price process is light-tailed. In both models, the tail exponent is determined by the least frequent trading type, while the proportions of faster-trading
types affect only the scale constant. We also provide sufficient conditions under which these Pareto-type tails sharpen to exact Paretian tails. These results identify random trade timing and heterogeneity in trading behavior
as a general mechanism for generating power-law tails in realized prices. 

\end{abstract}

\noindent\textbf{Keywords:} random trade timing; power-law tails; Markov-modulated L\'evy processes; generalized Erlang timing; realized prices.

\section{Introduction}
Power-law behavior in price changes has been documented for a wide range of commodities
and financial assets. Representative examples include stocks \citep{cont2001,gabaix2003,plerou2004},
foreign exchange \citep{mcfarland1982}, and spot and futures prices of several commodities
\citep{Mandelbrot1963,Matia2002}; for a broader perspective on power laws in economics and finance, see \citet{gabaix2009}. Although the empirical evidence is broad, theoretical mechanisms are often tied to particular markets or asset classes. A general question therefore remains open: how can power-law tails arise in realized prices when the underlying continuous-time price dynamics are light-tailed?

This paper studies one such mechanism based on random trade timing. Observed prices are realized only when a trade occurs, so the distribution of the realized price depends on both the latent price process and the random time at which the market is sampled. We model the realized price as
\begin{align}
P_T = e^{X_T},
\end{align}
where $X=(X_t,t\ge 0)$ is a Markov-modulated L\'{e}vy process and $T$ is the random time of the next trade. This formulation builds on the classical subordination idea of \citet{clark1973}, who modeled speculative prices as a stochastic process sampled at a random time.  Markov modulation allows the latent dynamics to vary across regimes, while random trade timing captures heterogeneity in trading behavior across market participants. The matrix Laplace exponent of $X$, denoted $A(\cdot)$, plays a central role in the analysis: its dominant real eigenvalue $r_D(A(\cdot))$ governs the moment generating function (MGF) of $X_t$ and, as shown below, the tail behavior of $P_T$.

We use the following terminology throughout; for general background on heavy-tailed distributions, see \citet{embrechtsetal1997} and \cite{resnick2007heavy}.
A random variable $Y$ has a \emph{Pareto-type tail} with exponent $\alpha>0$ if there exist constants $0<C_1\le C_2<\infty$ such that
\[
C_1 \le y^\alpha \PP(Y>y) \le C_2
\]
for all sufficiently large $y$, and a \emph{Paretian tail} with exponent $\alpha>0$ if $y^\alpha \PP(Y>y)$ converges to a positive constant as $y\to\infty$. A Paretian tail is a Pareto-type tail with a sharp asymptotic constant; the distinction matters because our results deliver only the former in general, with additional sufficient conditions required for the latter. More generally, we say that $Y$ has a Pareto-type tail with a \emph{logarithmic correction} of order $k\in\mathbb{N}$ if there exist constants $0<C_1\le C_2<\infty$ such that
\[
C_1 \le (\log y)^{-k} y^\alpha \PP(Y>y) \le C_2
\]
for all sufficiently large $y$.

We consider two trade-timing environments. In the intertrade-incidence model (IIM), trades occur on a discrete grid, and each type trades at each opportunity with its own probability $p$, so that the waiting time to the next trade is geometrically distributed. A natural extension, in which the waiting time follows a negative-binomial distribution, is also examined. In the intertrade-time model (ITM), the trade time is the sum of an exponential arrival time and a finite number of additional exponentially distributed stages: the exponential arrival captures the wait for the next potential counterparty, and the additional stages capture possible multi-stage transaction-completion delays that defer the realization of the next trade. The resulting distribution of $T$ is generalized Erlang, and the benchmark exponential-timing case of \citet{bst2018} arises as the special case in which no post-arrival stages are present. In both models, heterogeneity enters through a finite mixture of trading types, where each type is indexed by its own trading-frequency parameter---a trade probability $p$ in the IIM and an arrival intensity $\lambda$ in the ITM.

Our main results characterize the upper tail of $P_T$ in terms of the slowest
component of the trade-time distribution. In the IIM, this is the trading type
with the smallest trade probability. In the ITM, it is the smallest exponential
rate appearing in the trade-time distribution---either the smallest type-specific
arrival intensity or the smallest completion-stage rate, whichever is smaller.
In each case, the tail exponent is determined by an eigenvalue equation involving
the matrix Laplace exponent $A(\cdot)$ and this slowest rate. Faster components
affect the scale of the tail but not its exponent. Importantly, the mechanism can
produce Pareto-type upper tails in $P_T$ even when the latent price process is
light-tailed; in particular, this is the case when each regime has a nondegenerate
diffusion component, though a nondegenerate diffusion component is not necessary.
We also show that logarithmic corrections naturally arise when the waiting-time
mechanism contains higher-order components. In a negative-binomial extension of
the IIM, where the realized price is observed at the $n$-th successful trading
opportunity rather than the first one, the Pareto exponent is unchanged, but the
tail carries a logarithmic correction of order $n-1$. The ITM generates an
analogous phenomenon through generalized Erlang trade times and boundary cases
in which the smallest rate appears with multiplicity greater than one. Analogous
results hold for the lower tail of $P_T$, that is, for the probability of small
realized prices, by passing to the dual process $-X$.

The paper is related to existing explanations of heavy tails in specific markets,
including \cite{gabaix2003} and \cite{warusawitharana2018}. Our contribution differs
in emphasis. Rather than building a market-specific institutional model, we isolate
a general probabilistic channel through which random timing and heterogeneous trading
behavior produce heavy-tailed realized prices.

The remainder of the paper is organized as follows. Section \ref{smain} introduces the
Markov-modulated L\'{e}vy framework and the trade-timing models. Section \ref{siim}
studies the intertrade-incidence model, and Section \ref{sitm} studies the
intertrade-time model. The appendices collect the matrix and Tauberian results used
in the proofs.

\section{Model and preliminaries} \label{smain}
    \subsection{Model and assumptions} \label{smain1}
This subsection introduces the latent price process, the random trade-timing mechanisms, and the cross-sectional heterogeneity that drive the tail results.

\subsubsection{Latent price dynamics}  \label{smain1a}
Let $P_t$ denote the latent commodity price at time $t$. If $0 = T_0 < T_1 < \cdots < T_{K-1} < T_K$ are transaction times, we write the realized gross price change between two consecutive transactions as
\begin{align} \label{priceeq}
P_{T_k \mid T_{k-1}} = P_{T_{k-1}} e^{X_{T_k-T_{k-1}}}, \qquad k=1,\ldots,K,
\end{align}
where $(X_t,t\ge 0)$ is the log-price process. Since the paper is concerned with the tail behavior of realized price changes, we normalize $P_{T_{k-1}}=1$ whenever convenient and focus on the distribution of $e^{X_{T_k-T_{k-1}}}$.

Throughout, $(X_t,t\ge 0)$ is assumed to be a Markov-modulated L\'{e}vy process. We only briefly summarize the setup here; see \citet[Chapter XI]{asmussen2003} for a detailed treatment. Let $\mathcal E = \{1,\ldots,N\}$ and let $(X,J)=((X_t,J_t), t\in\mathbb R_+)$ be a bivariate Markov process on $\mathbb R \times \mathcal E$ with $P(X_0=0)=1$. The increments of $X$ are governed by the latent Markov chain $J$ in the sense that
\begin{align}
E\left[f(X_{t+s}-X_t) g(J_{t+s}) \mid \mathcal F_t\right]
= E\left[f(X_s-X_0) g(J_s) \mid J_0\right]
\label{maplaw}
\end{align}
for all $s,t\in\mathbb R_+$ and all nonnegative measurable functions $f$ and $g$. Let $R_t$ denote the transition matrix of $J$, and let $G=(g_{jk})$ denote its infinitesimal generator,
\[
G = \lim_{h\to 0+}\frac{R_h-I}{h}.
\]
We assume that $G$ is irreducible.

Conditional on $J_t=j$, the process $X$ evolves on $[t,t+s)$ as a L\'{e}vy process $X^{(j)}=(X_t^{(j)},t\ge 0)$ with exponent
\begin{align*}
\psi_j(z)
= \mu_j z + \frac{1}{2}\sigma_j^2 z^2 + \int_{-\infty}^{\infty}
\left(e^{zx}-1-zx\mathbbm{1}(|x|\le 1)\right)\nu_j(dx),
\end{align*}
for $z\in \mathcal S_j \equiv \{z\in\mathbb C : E[e^{\RRe(z)X_1^{(j)}}] < \infty\}$.

Additional jumps at transition epochs of $J$ are also allowed. Specifically, when $J$ jumps from $j$ to $k\neq j$, the process $X$ jumps with probability $\varrho_{jk}$, and the jump size is an iid random variable $U_{jk}$. We assume that $U_{jk}$ is independent of $J$ and of $(X^{(j)}, j\in\mathcal E)$, and that $\varrho_{jj}=0$ and $U_{jj}\equiv 0$ for all $j$. Thus a regime switch may or may not induce an additional jump in $X$; when $\varrho_{jk}=1$ and $U_{jk}$ is degenerate, one recovers deterministic level shifts at regime changes.

Let $\Psi(z)$ denote the diagonal matrix with diagonal entry $\psi_j(z)$ in state $j$, and let $\Upsilon(z)$ denote the matrix with $(j,k)$ entry $E[e^{zV_{jk}}]$, where $V_{jk}=U_{jk}$ with probability $\varrho_{jk}$ and $V_{jk}=0$ with probability $1-\varrho_{jk}$. Then Proposition 2.2 of Chapter XI in \citet{asmussen2003} implies that
\begin{align}
E\left[e^{zX_t}\mathbbm{1}(J_t=k) \mid J_0=j\right]
= \left(e^{A(z)t}\right)_{jk},
\qquad
A(z)=G\odot \Upsilon(z)+\Psi(z),
\label{assu1}
\end{align}
whenever $E[e^{zX_t}]$ exists. The matrix $A(z)$ is the matrix Laplace exponent of the Markov-modulated L\'{e}vy process. Because $G$ is irreducible, the same is true of $A(z)$ on its domain.

Markov modulation allows realized price changes to differ even when the elapsed time between transactions is the same. Without modulation, the law of $P_{T_k\mid T_{k-1}}$ would depend only on $T_k-T_{k-1}$. With modulation, the current regime matters as well:
\begin{align}
\PP(P_{T_k} \mid P_{T_{k-1}}=1, J_{T_{k-1}}=j)
= \PP(e^{X_{T_k-T_{k-1}}} \mid J_0=j).
\label{eqmmlevy}
\end{align}
Hence heterogeneity in realized price changes may arise both from the elapsed trading time and from the latent state at the beginning of the interval. For this reason, it is enough to study the one-step realized price
\begin{align} \label{priceeq0}
P_T = e^{X_T},
\end{align}
where the initial distribution of $J_0$ is an arbitrary probability vector $w_0$.

\subsubsection{Random trade timing} \label{smain1b}
Equation \eqref{assu1} describes the latent price dynamics, but the object of interest is the realized price observed when the next trade occurs. We therefore model the time to the next trade as a random variable $T$ and refer to its law as the intertrade-time distribution. Conceptually, the role of $T$ is analogous to a waiting-time distribution in queueing models: prices evolve continuously in latent time, but observation occurs only when a transaction arrives.

We consider two environments.
\begin{itemize}
    \item[(i)] In the intertrade-incidence model (IIM),
    \begin{align} \label{iim}
    T \sim \text{Geo}(p),
    \end{align}
where $\mathrm{Geo}(p)$ denotes the geometric distribution taking integer values $t\ge 1$ with probability $(1-p)^{t-1}p$.
    \item[(ii)] In the intertrade-time model (ITM),
    \begin{align} \label{itm}
    T = T_{\lambda} + \sum_{j=1}^{m-1} T_{\nu_j},
    \end{align}
    where $T_{\lambda} \sim \Exp(\lambda)$, $T_{\nu_j} \sim \Exp(\nu_j)$, and all components are independent.
\end{itemize}

In the IIM, trades occur on a discrete grid, and a trade takes place at each grid point with probability $p$. Thus the waiting time to the next trade is geometrically distributed. This specification is natural for markets in which trading opportunities are themselves discrete by institutional design---for example, periodic call auctions, commodity markets that open at fixed intervals, or order-book markets in which economically meaningful price changes occur only at scheduled clearing events. In such settings, each grid point represents the next moment at which a trade can take place, and no separate arrival-and-completion decomposition is needed.

In the ITM, $T_{\lambda}$ is the waiting time to the arrival of the next potential buyer, modeled as the first arrival time of a Poisson process with intensity $\lambda$. The additional term $\sum_{j=1}^{m-1} T_{\nu_j}$ captures potential post-arrival delays required to complete the transaction. Because the sum of independent Erlang variables is generalized Erlang
\citep{dehon1982}, this specification remains tractable even when the completion process consists of several heterogeneous exponential stages. The benchmark case studied in \citet{bst2018} is nested as the special case with no additional waiting stage, namely $\sum_{j=1}^{m-1} T_{\nu_j}=0$. The use of Erlang waiting times to capture multi-stage purchase or transaction processes is studied in the literature; see \citet{chatfield1973} for an early treatment of consumer purchasing with Erlang inter-purchase times.

\subsubsection{Heterogeneity across trading types} \label{smain1c}
The two trade-timing models in Section \ref{smain1b} are stated with a single trade probability $p$ in the IIM and a single arrival intensity $\lambda$ in the ITM. To accommodate cross-sectional heterogeneity, we replace these single parameters by a finite mixture over latent trading types.

Let $L\in\{1,\ldots,\tau\}$ denote the trading type, with
\[\PP(L=j)=q_j,\qquad j=1,\ldots,\tau,\]
where $q_j>0$ for every $j$. Conditional on $L=j$, the trade-time distribution is governed by a type-specific parameter, while all other model components are common across types. This mixture representation keeps the model tractable while allowing the observed distribution of trade times to reflect heterogeneity across buyers or market conditions.

In the IIM, heterogeneity operates through the trade probability $p$. Specifically, conditional on $L=j$, the trade probability is $p_j$, where
\[
\Pi=\{p_1,\ldots,p_\tau\},\qquad 0<p_1<\cdots<p_\tau<1.
\]
Thus, in the geometric IIM,
\begin{equation*} 
T\mid\{L=j\}\sim \mathrm{Geo}(p_j),
\end{equation*}
and hence the unconditional law $\mathcal L(T)$ of $T$ is the finite mixture
\begin{equation}\label{iimadd} 
\mathcal L(T)=\sum_{j=1}^{\tau}q_j\,\mathcal L(\mathrm{Geo}(p_j)).
\end{equation}
This finite-mixture representation is in line with the empirical marketing literature on heterogeneity in purchase-frequency models; see \citet{kaplon2010}. In Section \ref{siim} we also consider an extension of the IIM in which the geometric waiting time is replaced by a negative-binomial waiting time, retaining the same type structure with $p_j$ as the success probability for type $j$.

In the ITM, heterogeneity operates through the arrival intensity $\lambda$, while the post-arrival completion rates $(\nu_1,\ldots,\nu_{m-1})$ are common across types. Conditional on $L=j$, the arrival intensity is $\lambda_j$, where
\[
\Lambda=\{\lambda_1,\ldots,\lambda_\tau\},\qquad
0<\lambda_1<\cdots<\lambda_\tau.
\]
Thus
\[
T\mid\{L=j\} \stackrel{d}{=} T_{\lambda_j}+\sum_{k=1}^{m-1}T_{\nu_k},
\]
where $T_{\lambda_j}\sim\Exp(\lambda_j)$, $T_{\nu_k}\sim\Exp(\nu_k)$, and all components are independent. Consequently, the unconditional law $\mathcal L(T)$ of $T$ is again a finite mixture over types:
\begin{equation}\label{itmadd}
\mathcal L(T)=\sum_{j=1}^{\tau}q_j\,\mathcal L\!\left(T_{\lambda_j}+\sum_{k=1}^{m-1}T_{\nu_k}\right).
\end{equation}

\subsection{Preliminary results on Markov-modulated L\'{e}vy processes} \label{sec_pre_markov}
In this subsection, we introduce notation and briefly summarize useful results for the subsequent discussion.

Let $\Sigma(\mathbb{A})$ denote the spectrum of a square matrix $\mathbb{A}$. We define
\begin{align}
&\rho(\mathbb{A}) = \max_j\{ |r_j| : r_j \in \Sigma(\mathbb{A}) \}, \\
&\tau(\mathbb{A}) =\max_j\{ \RRe(r_j) : r_j \in \Sigma(\mathbb{A}) \},
\end{align}
where $\rho(\mathbb{A})$ (resp.\ $\tau(\mathbb{A})$) is called the spectral radius (resp.\ spectral abscissa) of $\mathbb{A}$. We also let $r_D(\mathbb{A})$ denote the real eigenvalue of $\mathbb{A}$ that equals $\tau(\mathbb{A})$ when such an eigenvalue exists; we call $r_D(\mathbb{A})$ the dominant real eigenvalue. Moreover, we say that a square matrix $\mathbb{A}$ is Metzler if and only if $\upsilon I + \mathbb{A}$ is a nonnegative matrix for some $\upsilon$. Appendix \ref{appmetzler} collects properties of Metzler matrices that are repeatedly used in the subsequent discussion.

For convenience, we define the following set on which the MGF of $X_t$ is well defined for all $t>0$:
\begin{align}
\mathcal I = \left\{ s \in \mathbb{R} :  E\left[e^{sX_t} \mid J_0 = j\right] < \infty, \quad \forall j \in \mathcal E \right\}.
\end{align}

We introduce a preliminary result that characterizes the MGF of a Markov-modulated L\'{e}vy process on $\mathcal I$.
\begin{lemma}\label{lemmain}
Let $(X_t, t\geq 0)$ be a Markov-modulated L\'{e}vy process with the matrix Laplace exponent $A(z)$, described in Section \ref{smain1a}. Let $1_N$ denote the $N \times 1$ vector of ones, and let $w_0=(w_{01},w_{02},\ldots,w_{0N})'$ denote the initial distribution of $J$, i.e., $w_{0,j}=\PP(J_0 = j)$ for $j = 1,\ldots,N$. Then
\begin{align} \label{eqlem001}
M_t(s) \coloneqq E[e^{sX_t}]= w_0' \, e^{A(s)t} \, 1_N, \qquad s \in \mathcal I.
\end{align}
\end{lemma}

\begin{proof}
For $s \in \mathcal I$, $E[e^{sX_t} \mid J_0 = j] < \infty$ for all $j$. Since the $(j,k)$-th element of $e^{A(s)t}$ is $E[e^{s X_t}\mathbbm{1}(J_t=k) \mid J_0 = j]$, we obtain
\begin{align}
\left( \begin{matrix} E\left[e^{sX_t} \mid J_0 = 1 \right] \\ \vdots \\ E\left[e^{sX_t} \mid J_0 = N \right]  \end{matrix}\right) = e^{A(s)t} \, 1_N. \label{eqcondi1}
\end{align}
From \eqref{eqcondi1}, \eqref{eqlem001} is easily established.
\end{proof}

The next lemma collects properties of the dominant real eigenvalue $r_D(A(s))$ that will be used repeatedly in Sections \ref{siim} and \ref{sitm}.
\begin{lemma}\label{lemmain2}
Let the assumptions in Lemma \ref{lemmain} hold. The Laplace exponent $A(s)$ of $(X_t, t\geq 0)$ satisfies the following:
\begin{enumerate}[label=(\roman*), ref=(\roman*)]
    \item \label{lemmain2a} $r_D(A(s))$ exists for all $s \in \mathcal I$.
    \item \label{lemmain2b} $r_D(A(0)) = 0$.
    \item \label{lemmain2c} $r_D(A(s))$ is convex in $s \in \mathcal I$.
\end{enumerate}
\end{lemma}
\begin{proof}
Note that $A(s)$ is Metzler. Then from Lemma \ref{metzlerthm}-\ref{metzlerthma}, we may easily deduce that the dominant real eigenvalue exists for all $s \in \mathcal I$; this proves \ref{lemmain2a}.

Observe that $A(0) = G$. Since $G$ is an irreducible generator matrix, $0$ is an eigenvalue of $G$ and all other eigenvalues have strictly negative real parts. Therefore, $0$ is the dominant real eigenvalue of $A(0)$, which proves \ref{lemmain2b}.

\ref{lemmain2c} is a direct consequence of Proposition 3.1 in \citet{bst2018}, given that $r_D(A(z))$ coincides with the spectral abscissa $\tau(A(z))$ of $A(z)$.
\end{proof}

Throughout, we will employ the following assumption on the Laplace exponent $A(z)$.
\begin{assumption}[Holomorphy around the candidate pole]\label{asshol}
Whenever $\alpha>0$ solves an equation of the form $r_D(A(-\alpha))=c$ for some constant $c>0$ used below, the matrix exponent $A(z)$ admits a holomorphic extension to an open connected neighborhood of the real interval $[-\alpha,0]$.
\end{assumption}

\begin{remark}
Assumption \ref{asshol} is imposed only to justify the simple-pole argument in the subsequent discussion (see Appendix \ref{appsimpole}). In particular, it ensures that the matrix-valued functions constructed there are holomorphic on a complex neighborhood of the candidate pole. This condition is not restrictive; a sufficient condition, which can replace Assumption \ref{asshol} in the subsequent discussion, is the existence of an open strip
\[
S=\{z\in\mathbb C : a< \RRe(z)<b\}, \qquad [-\alpha,0]\subset S,
\]
such that $S$ is contained in the domain of each regime-specific L\'evy exponent
$\psi_j$ and each transition-jump transform
\[
\upsilon_{jk}(z):=E[e^{zV_{jk}}], \qquad j,k\in\mathcal E.
\]
Then each entry of $\Upsilon(z)$ and each diagonal entry of $\Psi(z)$ is holomorphic on $S$. Since $\mathbb C^{N\times N}$ is finite dimensional, a matrix-valued function is holomorphic if and only if each of its entries is holomorphic. Hence $A(z)=G\odot \Upsilon(z)+\Psi(z)$ is holomorphic on $S$. Because $S$ is open and connected and contains the real interval $[-\alpha,0]$, Assumption \ref{asshol} follows.
\end{remark}
\section{Tail behavior in the intertrade-incidence model} \label{siim}

We first consider the IIM. The main specification is the geometric waiting time in \eqref{iim}. After establishing the geometric case, we extend the analysis to a negative-binomial waiting time, in which the realized price is observed at the $n$-th successful trading opportunity rather than the first. As we show, the same boundary equation determines the tail exponent, but the simple pole at the boundary becomes a pole of order $n$, and the tail acquires a logarithmic correction of order $n-1$.

Define
\begin{align} \label{iimset}
\mathcal I_{p_{1}} &= \left\{s \in \mathbb{R} : r_D(A(s)) < -\log(1-p_{1}) \right\}, \\
\mathcal S_{p_{1}} &= \left\{ z \in \mathbb{C} : \RRe(z) \in \mathcal I_{p_{1}} \right\}.
\end{align}
Clearly, $\mathcal I_{p_{1}} \subset \mathcal I$, and $M_t(s)$ is well defined for any $s \in \mathcal I_{p_{1}}$. Moreover, since
\[
\left|E\left[ e^{-z X_t}\right]\right| \le E\left[\left| e^{-z X_t}\right|\right] = E\left[ e^{-\RRe(z) X_t} \right],
\]
$M_t(z)$ is well defined and its modulus is bounded by $M_t(\RRe(z))$ for any $z \in \mathcal S_{p_{1}}$.

Our strategy to derive the tail behavior of $P_T$ is as follows. We first show that $M_T(s)$ is well defined on $\mathcal I_{p_{1}}$ and has a pole at the boundary point $s=-\alpha$, where $r_D(A(-\alpha)) = -\log(1-p_{1})$. Equivalently, $M_T(z)$ is well defined on $\mathcal S_{p_{1}}$ but has a pole somewhere on the line $\RRe(z)=-\alpha$. We then use the Tauberian results in Appendix \ref{apptauberian} to characterize the tail behavior of $X_T$, from which that of $P_T$ follows.

Following \citet{bst2018}, we say that an MGF is \emph{nonlattice} if it corresponds to a distribution that is not supported on a lattice.

\begin{proposition} \label{propmain}
Let $(X_t, t\ge 0)$ be a Markov-modulated L\'{e}vy process described in Section \ref{smain1a}, and let the trade time $T$, described in Sections \ref{smain1b}--\ref{smain1c}, satisfy \eqref{iim} and \eqref{iimadd}. Suppose that Assumption \ref{asshol} holds and that there exists a positive real number $\alpha$ such that $r_D(A(-\alpha)) = -\log(1-p_{1})$. Then $\alpha$ is unique, and the following hold.
\begin{enumerate}[label=(\roman*), ref=(\roman*)]
\item \label{propmaina} There exist constants $0<C_1\le C_2<\infty$ such that
\begin{align} \label{propmaineq1}
\widetilde{M}C_1 \le y^\alpha \PP(P_T>y) \le \widetilde{M}C_2
\end{align}
for all sufficiently large $y$; that is, $P_T$ has a Pareto-type tail.
\item \label{propmainb} Suppose, in addition, that either
\begin{enumerate}[label=(\alph*)]
\item the MGF $z\mapsto e^{\psi_j(z)}$ is nonlattice for some $j\in\mathcal E$, or
\item the transition-jump MGF $z\mapsto E[e^{zV_{jk}}]$ is nonlattice for some $j,k\in\mathcal E$ with $\varrho_{jk}>0$.
\end{enumerate}
In particular, condition (a) holds whenever $\sigma_j^2>0$ for some $j\in\mathcal E$. Then
\[
\lim_{y \to \infty} y^\alpha \PP(P_T>y) = \widetilde M/\alpha.
\]
\end{enumerate}
In \ref{propmaina} and \ref{propmainb}, $\widetilde{M}$ is positive and given by
\[
\widetilde{M} = \left(\frac{q_{1}p_{1}}{1-p_{1}}\right) \frac{(w_0' x_{p})(y_{p}'1_N)}{-\,y_{p}' B^{(1)}(-\alpha)x_{p}},
\]
where $B(s) = e^{\log(1-p_{1})I+A(s)}$, $B^{(1)}(s) = \frac{d}{ds}B(s)$, and $x_p$, $y_p$ are right and left eigenvectors of $B(-\alpha)$ associated with the unit eigenvalue.
\end{proposition}

\begin{proof}
The uniqueness of $\alpha$ satisfying $r_D(A(-\alpha)) = -\log(1-p_{1})$ follows from the convexity of $r_D(A(s))$ established in Lemma \ref{lemmain2}-\ref{lemmain2c}, together with $r_D(A(0))=0<-\log(1-p_1)$. Indeed, the sublevel set $\{s\in\mathcal I : r_D(A(s))\le -\log(1-p_1)\}$ is an interval containing $0$, so there can be at most one solution on the negative half-line, equivalently at most one positive $\alpha$.

\smallskip\noindent\textbf{Proof of \ref{propmaina}.}
Conditioning first on the latent trading type and then on the trade time, we obtain
\begin{align}
M_T(s) = E\left[e^{sX_T}\right] &= \sum_{j=1}^{\tau} q_j \, E\left[e^{sX_T}\mid p=p_j\right] \notag\\
&= \sum_{j=1}^{\tau} q_j \sum_{t=1}^{\infty} \PP(T=t\mid p=p_j)\, w_0' e^{A(s)t} 1_N \notag\\
&= \sum_{j=1}^{\tau} q_j \sum_{t=1}^{\infty} (1-p_j)^{t-1}p_j\, w_0' e^{A(s)t} 1_N \notag\\
&= \sum_{j=1}^{\tau}\frac{q_jp_j}{1-p_j}\, w_0' \left(\sum_{t=1}^{\infty}(1-p_j)^t e^{A(s)t}\right)1_N.
\label{eqinteg11}
\end{align}
The infinite sum $\sum_{t=1}^{\infty}(1-p_j)^t e^{A(s)t}$ in \eqref{eqinteg11} can be written as
\begin{align}
\sum_{t=1}^{\infty} e^{t\log(1-p_j)} e^{A(s)t} = \sum_{t=1}^{\infty} e^{(\log(1-p_j)I+A(s))t}. \label{eqinteg22}
\end{align}
Since $r_D(\log(1-p_j)I+A(s)) = \log(1-p_j) + r_D(A(s)) < 0$, all eigenvalues of $\log(1-p_j)I+A(s)$ have negative real parts for every $s \in \mathcal I_{p_{1}}$ and every $j = 1,\ldots,\tau$. This implies that $e^{(\log(1-p_j)I+A(s))t} \to 0$ as $t\to \infty$. Note that
\[
\bigl(I- e^{\log(1-p_j)I+A(s)}\bigr) \sum_{t=0}^{\ell} e^{(\log(1-p_j)I+A(s))t} = I - e^{(\log(1-p_j)I+A(s))(\ell+1)},
\]
where the right-hand side converges to $I$ as $\ell \to \infty$. Therefore
\begin{align} \label{eqinteg33}
\eqref{eqinteg22} = \left(I - e^{\log(1-p_j)I+A(s)}\right)^{-1} - I,
\end{align}
and hence
\begin{align} \label{eqinteg44}
M_T(s) = \sum_{j=1}^\tau \frac{q_jp_j}{1-p_j} w_0' \left(I- e^{\log(1-p_j)I+A(s)}\right)^{-1}1_N - \sum_{j=1}^\tau \frac{q_jp_j}{1-p_j} w_0' 1_N.
\end{align}

The inverse appearing in the first term of \eqref{eqinteg44} exists if and only if $e^{\log(1-p_j)I+A(s)}$ does not have a unit eigenvalue. By the spectral mapping theorem,
\[
\Sigma\left(e^{\log(1-p_j)I+A(s)}\right) = \exp\left(\Sigma\left(\log(1-p_j)I+A(s)\right)\right),
\]
where, as before, $\Sigma(\cdot)$ denotes the spectrum of its argument. Thus the inverse fails to exist precisely when $-\log(1-p_j)$ is an eigenvalue of $A(s)$ for some $j$. On $\mathcal I_{p_{1}}$, this never occurs. However, at $s=-\alpha$, the value $-\log(1-p_1)$ is an eigenvalue of $A(-\alpha)$ by hypothesis, so $M_T(s)$ has a pole at $s=-\alpha$.

To verify that this pole is simple, note from Assumption \ref{asshol} that $A(-z)$ is holomorphic on an open connected neighborhood of $[0,\alpha]$. Define
\[
F(z) \coloneqq e^{\log(1-p_1)I + A(-z)} - I, \qquad z \in [0,\alpha].
\]
For real $z$, the matrix $e^{A(-z)}$ has strictly positive entries (see, e.g., Theorem 3.12 in Chapter 6 of \citet{berman1994}); hence $F(z)$ is irreducible Metzler. By Lemma \ref{lemmain2}-\ref{lemmain2a}, the spectral mapping theorem, and the Perron--Frobenius theorem,
\[
r_D(F(z)) = e^{\log(1-p_1)+r_D(A(-z))}-1.
\]
Note that $z \mapsto r_D(A(-z))$ is convex (Lemma \ref{lemmain2}-\ref{lemmain2c}), and the map $x \mapsto e^{\log(1-p_1)+x}-1$ is strictly increasing and convex. These imply that $z\mapsto r_D(F(z))$ is also convex. Moreover,
\[
r_D(F(0)) = e^{\log(1-p_1)+r_D(A(0))}-1 = -p_1 < 0, \qquad r_D(F(\alpha)) = e^{\log(1-p_1)+r_D(A(-\alpha))}-1 = 0.
\]
Hence Lemma \ref{lemconvsimple} applies, showing that $F(z)^{-1}$ has a simple pole at $z=\alpha$. Since
\[
F(-s) = e^{\log(1-p_1)I+A(s)}-I = -\left(I-e^{\log(1-p_1)I+A(s)}\right),
\]
it follows that $\left(I-e^{\log(1-p_1)I+A(s)}\right)^{-1}$ has a simple pole at $s=-\alpha$.

Let $x_p$ and $y_p$ denote right and left eigenvectors of $B(-\alpha)$ associated with the unit eigenvalue. The residue is
\begin{align}
R \coloneqq \lim_{s\to -\alpha}(s+\alpha)\left(I-e^{\log(1-p_1)I+A(s)}\right)^{-1}
= x_p\left(-\,y_p' B^{(1)}(-\alpha)x_p\right)^{-1}y_p'
= c_p x_p y_p',
\end{align}
where $B^{(1)}(s)=\frac{d}{ds}e^{\log(1-p_1)I+A(s)}$ and $c_p = \left(-\,y_p' B^{(1)}(-\alpha)x_p\right)^{-1} > 0$; see Lemma \ref{lemsimpole}. Moreover, for all $j$ with $p_j > p_{1}$, $\left(I-e^{\log(1-p_{j})I+A(s)}\right)^{-1}$ is well defined and holomorphic at $s=-\alpha$ (see Lemma \ref{metzlerthm}-\ref{metzlerthmc}). Therefore
\begin{align}
\widetilde{M} \coloneqq \lim_{s\to -\alpha}(s+\alpha)M_T(s)
= \frac{q_{1}p_{1}}{1-p_{1}} (w_0' R 1_N)
= \frac{q_{1}p_{1}}{1-p_{1}} c_p (w_0' x_p) (y_p'1_N). \label{resideq1}
\end{align}
Since $x_p$ and $y_p$ have strictly positive entries and $w_0\neq 0$ has nonnegative entries, the right-hand side of \eqref{resideq1} is positive. Hence $M_T(s)$ has a simple pole at $s=-\alpha$.

It then follows from Corollary \ref{cortau} that there exist constants $C_1$ and $C_2$ satisfying
\begin{align}
\widetilde{M} C_1 \le \liminf_{x \to \infty} e^{\alpha x} \PP(X_T>x) \le \limsup_{x \to \infty} e^{\alpha x} \PP(X_T>x) \le \widetilde{M} C_2. \label{eqtau11}
\end{align}
That is, for any $\varepsilon > 0$, there exists $x_0$ such that
\begin{align}
\widetilde{M} C_1 - \varepsilon < e^{\alpha x} \PP(X_T>x) < \widetilde{M} C_2 + \varepsilon \label{eqtau22}
\end{align}
for all $x \ge x_0$. Substituting $y=e^x$ yields
\begin{align}
\widetilde{M} C_1 - \varepsilon < y^\alpha \PP(P_T>y) < \widetilde{M} C_2 + \varepsilon, \label{eqtau33}
\end{align}
which establishes \eqref{propmaineq1}.

\smallskip\noindent\textbf{Proof of \ref{propmainb}.}
Under either of the stated nonlattice conditions, the argument of Theorem 3.2 in \citet{bst2018} implies that
\[
\tau(A(-\alpha+\ii\beta)) < \tau(A(-\alpha)) \qquad \text{for all }\beta\neq 0.
\]
Therefore
\[
\rho\!\left((1-p_1)e^{A(-\alpha+\ii\beta)}\right) = (1-p_1)e^{\tau(A(-\alpha+\ii\beta))} < (1-p_1)e^{\tau(A(-\alpha))} = 1,
\]
so $I-(1-p_1)e^{A(-\alpha+\ii\beta)}$ is invertible for every $\beta\neq 0$. Hence $s=-\alpha$ is the unique singularity of $M_T(s)$ on its axis of convergence, and Corollary \ref{cortau} yields the desired result.
\end{proof}

Proposition \ref{propmain} shows that, for a latent Markov-modulated L\'{e}vy process with matrix Laplace exponent $A(s)$, the tail exponent $\alpha$ is characterized by the equation $r_D(A(-\alpha)) = -\log(1-p_{1})$. Hence the slowest-trading type---the type with the smallest trade probability $p_1$---determines the rate at which the upper tail decays. The corresponding population share $q_{1}$ enters the scale constant $\widetilde{M}$, while faster-trading types influence the asymptotic tail only through that scale. In this sense, the lower tail of the trading-frequency distribution is the key determinant of the upper tail of realized prices in the IIM.

The upper-tail results of this section can be extended to the lower-tail probability $\PP(P_T<y^{-1})$ with only straightforward modifications. Since an analogous extension applies to the ITM as well, we postpone the discussion to Remark \ref{remlowertail}.

\begin{example}
Suppose that $(X_t, t \ge 0)$ is a Brownian motion with drift. Then $A(s)=\psi(s)$, where
\[
\psi(s) = \mu s + \frac{1}{2}\sigma^2 s^2,
\]
so the dominant real eigenvalue is simply $r_D(A(s))=\psi(s)$. The equation determining the tail exponent becomes
\[
\psi(-\alpha) = -\log(1-p_1),
\]
which is quadratic in $\alpha$. If $\sigma^2>0$, the unique positive solution is
\[
\alpha = -\mu/\sigma^2 + \sqrt{\mu^2 - 2\sigma^2 \log(1-p_1)}/\sigma^2,
\]
and Proposition \ref{propmain}-\ref{propmainb} then implies
\[
y^{\alpha}\PP(P_T>y) \to \widetilde{M}/\alpha.
\]
If $\sigma^2=0$, a positive solution exists if and only if $\mu<0$, in which case $\alpha=\log(1-p_1)/\mu$. Proposition \ref{propmain}-\ref{propmaina} then yields, for sufficiently large $y$,
\begin{equation} \label{propmainex2}
\widetilde{M}C_1 \le y^\alpha \PP(P_T>y) \le \widetilde{M}C_2, \qquad \alpha = \log(1-p_1)/\mu.
\end{equation}
\end{example}

\subsection*{A negative-binomial extension}
We now consider an extension of the IIM in which the geometric waiting time in \eqref{iim} is replaced by a negative-binomial waiting time. Fix an integer $n\ge 1$, and suppose that conditional on $p=p_j$ (i.e., $L=j$),
\begin{equation}\label{eqnbinom}
\PP(T=t\mid p=p_j) = \binom{t-1}{n-1}p_j^n(1-p_j)^{t-n}, \qquad t=n,n+1,\ldots.
\end{equation}
Equivalently, $T$ is the index of the $n$-th successful trading opportunity on the discrete grid; the geometric IIM is the special case $n=1$.

\begin{cor}[Negative-binomial intertrade incidence]\label{cornegbin}
Maintain the assumptions on $X$ and the type distribution, and Assumption \ref{asshol}, in Proposition \ref{propmain}, but replace the conditional geometric law of $T$ by the negative-binomial law \eqref{eqnbinom}. Suppose that there exists $\alpha>0$ such that $r_D(A(-\alpha))=-\log(1-p_1)$. Then $\alpha$ is unique, and there exist constants $0<C_1\le C_2<\infty$ such that, for all sufficiently large $y$,
\[
\widetilde M_n C_1 \le (\log y)^{-(n-1)}y^\alpha \PP(P_T>y) \le \widetilde M_n C_2,
\]
where $\widetilde M_n>0$ is given by
\[
\widetilde M_n = q_1\left(\frac{p_1}{1-p_1}\right)^{n}\bigl(-\,y_{p}' B^{(1)}(-\alpha)x_{p}\bigr)^{-n}(y_p'x_p)^{n-1}(w_0'x_p)(y_p'1_N),
\]
with $B(s) = e^{\log(1-p_{1})I+A(s)}$, $B^{(1)}(s) = \frac{d}{ds}B(s)$, and $x_p$, $y_p$ right and left eigenvectors of $B(-\alpha)$ associated with the unit eigenvalue. In particular, the negative-binomial IIM has the same Pareto exponent as the geometric IIM, but for $n\ge 2$ the tail carries a logarithmic correction of order $n-1$.
\end{cor}

\begin{proof}
Uniqueness of $\alpha$ follows from the same convexity argument as in Proposition \ref{propmain}. Define
\[
\widetilde B_j(s) \coloneqq (1-p_j)\,e^{A(s)}, \qquad j=1,\ldots,\tau,
\]
so that $\widetilde B_1(s) = B(s)$. For $s\in\mathcal I_{p_1}$, the spectral radius of $\widetilde B_j(s)$ is strictly less than one, so the matrix series $\sum_{u\ge 0}\binom{u+n-1}{n-1}\widetilde B_j(s)^u$ converges to $(I-\widetilde B_j(s))^{-n}$. Conditional on $p=p_j$,
\begin{align}
E[e^{A(s)T}\mid p=p_j]
&= \sum_{t=n}^{\infty}\binom{t-1}{n-1}p_j^n(1-p_j)^{t-n}e^{A(s)t} \notag\\
&= (p_j e^{A(s)})^n \sum_{u=0}^{\infty}\binom{u+n-1}{n-1}\widetilde B_j(s)^u \notag\\
&= (p_j e^{A(s)})^n (I-\widetilde B_j(s))^{-n}.
\label{eqnb1}
\end{align}
Using $e^{A(s)} = (1-p_j)^{-1}\widetilde B_j(s)$ together with the matrix identity $\widetilde B_j(s)(I-\widetilde B_j(s))^{-1} = (I-\widetilde B_j(s))^{-1}-I$ (valid since $\widetilde B_j(s)$ and $(I-\widetilde B_j(s))^{-1}$ commute), \eqref{eqnb1} can be rewritten as
\[
E[e^{A(s)T}\mid p=p_j] = \left(\frac{p_j}{1-p_j}\right)^{n}\bigl((I-\widetilde B_j(s))^{-1}-I\bigr)^{n}.
\]
Hence
\begin{equation}\label{eqnbMT}
M_T(s) = \sum_{j=1}^{\tau}q_j\left(\frac{p_j}{1-p_j}\right)^{n} w_0'\bigl((I-\widetilde B_j(s))^{-1}-I\bigr)^{n}1_N.
\end{equation}

As in the proof of Proposition \ref{propmain}, $I-\widetilde B_j(-\alpha)$ is invertible for every $j>1$, while $(I-\widetilde B_1(s))^{-1} = (I-B(s))^{-1}$ has a simple pole at $s=-\alpha$ with residue
\[
R = c_p x_p y_p', \qquad c_p = \bigl(-\,y_p' B^{(1)}(-\alpha)x_p\bigr)^{-1}>0.
\]
Consequently,
\[
(I-\widetilde B_1(s))^{-1} = \frac{R}{s+\alpha} + H(s),
\]
where $H(s)$ is holomorphic at $s=-\alpha$, so
\[
\bigl((I-\widetilde B_1(s))^{-1}-I\bigr)^{n} = \frac{R^{n}}{(s+\alpha)^{n}} + \text{lower-order terms}.
\]
Since $R^{n}=c_p^{n}(y_p'x_p)^{n-1}x_p y_p'$, and $y_p'x_p>0$ by the Perron--Frobenius positivity of $x_p$ and $y_p$, the leading Laurent coefficient of $M_T$ at $s=-\alpha$ is
\begin{align}
\widetilde M_n
&= \lim_{s\to -\alpha}(s+\alpha)^{n}M_T(s) \notag\\
&= q_1\left(\frac{p_1}{1-p_1}\right)^{n} c_p^{n}(y_p'x_p)^{n-1}(w_0'x_p)(y_p'1_N) \;>\; 0,
\label{eqnbresidue}
\end{align}
which matches the formula in the statement after substituting $c_p=(-\,y_p' B^{(1)}(-\alpha)x_p)^{-1}$. Thus $M_T(s)$ has a pole of order exactly $n$ at $s=-\alpha$.

The conclusion now follows from Corollary \ref{cortau} with $\ell=n$, exactly as in the final step of the proof of Proposition \ref{propmain}: there exist constants $C_1, C_2>0$ such that
\[
\widetilde M_n C_1 \le x^{-(n-1)}e^{\alpha x}\PP(X_T>x) \le \widetilde M_n C_2
\]
for all sufficiently large $x$, and substituting $y=e^x$ gives the stated bound.
\end{proof}

Corollary \ref{cornegbin} shows that logarithmic corrections, which make the tail of $P_T$ thicker, can already arise in a discrete-grid incidence model once the realized price is observed at the $n$-th successful trading opportunity rather than the first. In the transform calculation, the geometric waiting time contributes a single inverse $(I-B(s))^{-1}$, whereas the negative-binomial waiting time contributes its $n$-th power. This converts a simple pole into a pole of order $n$, and the Tauberian theorem translates the higher-order pole into the factor $(\log y)^{n-1}$ in the tail. Intuitively, observing $P_T$ only at the $n$-th successful opportunity introduces additional uncertainty about its value relative to the geometric IIM; that this added uncertainty actually reshapes the tail of $P_T$, however, is not obvious a priori, and is the content of Corollary \ref{cornegbin}. An analogous logarithmic correction reappears in the ITM in Section \ref{sitm}, where the trade-time distribution is generalized Erlang and boundary cases of the type structure produce comparable logarithmic factors.

   \section{Tail behavior in the intertrade-time model}\label{sitm}

In this section, we consider the tail behavior of $P_T$ under the ITM. Let $\lambda_{\min} = \min\{\lambda_1, \mu_1\}$ and define
\begin{align} \label{iimset2}
\mathcal I_{\lambda_{\min}} &= \left\{s \in \mathbb{R} : r_D(A(s)) < \lambda_{\min} \right\}, \\
\mathcal S_{\lambda_{\min}} &= \left\{ z \in \mathbb{C} : \RRe(z) \in \mathcal I_{\lambda_{\min}} \right\}.
\end{align}
We obtain the tail behavior of $P_T$ via an approach similar to that of Section \ref{siim}.

Under cross-sectional heterogeneity, $T$, conditional on $\lambda=\lambda_j$ (i.e., $L=j$), is the sum of $m$ independent exponential random variables $T_{\lambda_j}$ and $T_{\nu_h}$ for $h=1,\ldots,m-1$. It is possible that $\nu_h=\nu_{h'}$ for some $h\neq h'$. Let $(\mu_k, k=1,\ldots,d)$ be the distinct elements of $(\nu_h, h=1,\ldots,m-1)$, ordered so that $\mu_1<\cdots<\mu_d$, and let $r_k$ denote the multiplicity of $\mu_k$ in $(\nu_h, h=1,\ldots,m-1)$. For each type $j=1,\ldots,\tau$, let $T_j$ denote a random variable with the conditional law of $T$ given $\lambda=\lambda_j$. Then
\[
T_j \stackrel{d}{=} T_{\lambda_j}+\sum_{k=1}^{d}\Erl(\mu_k,r_k),
\]
where all summands are independent and $\Erl(\mu_k,r_k)$ denotes the Erlang distribution with rate parameter $\mu_k$ and shape parameter $r_k$. The law of $T$, denoted $\mathcal L(T)$, is then the finite mixture
\begin{align} \label{itm2}
\mathcal L(T) = \sum_{j=1}^{\tau}q_j\,\mathcal L(T_j).
\end{align}
Since the exponential variable $T_{\lambda_j}$ is itself $\Erl(\lambda_j,1)$, $T_j$ is a generalized Erlang random variable given by a sum of Erlang random variables with rate sequence $(\lambda_j,\mu_1,\ldots,\mu_d)$ and shape sequence $(1,r_1,\ldots,r_d)$. The subsequent analysis requires the probability density function of these sums.

If all rates are distinct, the generalized Erlang density expansion of Section \ref{sec_erlang} applies directly. If $\lambda_j$ coincides with one of the completion-stage rates, the equal-rate stages are first merged, and the same expansion is then applied. More explicitly, if $\lambda_j\neq \mu_k$ for all $k=1,\ldots,d$, the results of Section \ref{sec_erlang} imply that the probability density $f_{T_j}$ of $T_j$ is
\begin{equation}\label{eqerlang001}
f_{T_j}(t) = c_{1,1,j}\,e^{-\lambda_j t} + \sum_{k=2}^{d+1}\sum_{\ell=1}^{r_{k-1}} \frac{c_{k,\ell,j}}{(\ell-1)!}\, t^{\ell-1}e^{-\mu_{k-1}t}, \qquad t>0.
\end{equation}
If $\lambda_j=\mu_{k^\ast}$ for some $k^\ast\in\{1,\ldots,d\}$, then $T_j\stackrel{d}{=}\Erl(\mu_{k^\ast},r_{k^\ast}+1) + \sum_{k\neq k^\ast}\Erl(\mu_k,r_k)$, and the same results yield
\begin{equation}\label{eqerlang002}
f_{T_j}(t) = \sum_{\ell=1}^{r_{k^\ast}+1} \frac{\tilde c_{k^\ast,\ell,j}}{(\ell-1)!}\, t^{\ell-1}e^{-\mu_{k^\ast}t} + \sum_{\substack{k=1\\ k\neq k^\ast}}^{d} \sum_{\ell=1}^{r_k} \frac{\tilde c_{k,\ell,j}}{(\ell-1)!}\, t^{\ell-1}e^{-\mu_k t}, \qquad t>0.
\end{equation}
Here $c_{k,\ell,j}$ and $\tilde c_{k,\ell,j}$ denote real coefficients; for each $j$, their explicit expressions follow from \eqref{coeffc} with the appropriate rate sequence $(a_1,\ldots,a_D)$ and shape sequence $(b_1,\ldots,b_D)$. For brevity, we refer to $c_{k,\ell,j}$ and $\tilde c_{k,\ell,j}$ as the \emph{Erlang coefficients} of $T_j$.

Because the density of $T_j$ takes different forms depending on whether $\lambda_j$ coincides with one of the completion-stage rates, it is convenient to define
\[
\mathcal J^* := \{j\in\{1,\ldots,\tau\} : \lambda_j=\mu_k \text{ for some } k\in\{1,\ldots,d\}\}.
\]
For $j\notin\mathcal J^*$, $T_j$ has $d+1$ distinct rates $(\lambda_j,\mu_1,\ldots,\mu_d)$, and its density is given by \eqref{eqerlang001}. For $j\in\mathcal J^*$, let $u(j)$ denote the unique index such that $\lambda_j=\mu_{u(j)}$; the equal-rate stages are then merged, and the density is given by \eqref{eqerlang002} with $k^\ast=u(j)$.

\begin{proposition} \label{propmain2e}
  Let $(X_t, t\geq 0)$ be a Markov modulated L\'{e}vy process, described in Section \ref{smain1a} and let the trade time $T$, described in Sections \ref{smain1b}-\ref{smain1c}, satisfy \eqref{itm} and \eqref{itmadd}. 
Suppose that Assumption \ref{asshol} holds and there exists $\alpha>0$ such that $r_D(A(-\alpha)) = \lambda_{\min}$.
Then such $\alpha$ is unique, and the following hold.
\begin{enumerate}[label=(\roman*), ref=(\roman*)]
\item\label{propmain2ea} There exist constants \(0<C_1\le C_2<\infty\) such that
\begin{align} \label{propmain2eeq3}
\widetilde M C_1 \le (\log y)^{-\beta} y^\alpha \PP(P_T>y) \le \widetilde M C_2 
\end{align}
for sufficiently large $y$, where $\beta$ is determined by the following case distinction:
 \begin{enumerate}[label=(\alph*), ref=(\alph*)]
    \item \label{propmain2enum1} If $\lambda_{\min} = \lambda_1 < \mu_1$, then $\beta = 0$.
    \item\label{propmain2enum2}  If $\lambda_{\min}  = \mu_1 < \lambda_1$, then $\beta = r_1-1$.
    \item\label{propmain2enum3}  If $\lambda_{\min}  = \lambda_1 =  \mu_1$, then $\beta = r_1$.
\end{enumerate}
    \item  \label{propmain2eb}  Suppose, in addition, the nonlattice condition in Proposition \ref{propmain}\,\ref{propmainb} holds, and $\lambda_{\min}=\lambda_1 < \mu_1$ or $\lambda_{\min}  = \mu_1 < \lambda_1$ with $r_1 = 1$. Then,  as $y\to\infty$,
\[
\lim_{y \to \infty}  y^\alpha \PP(P_T>y) = \widetilde M/\alpha.
\]
\end{enumerate}
In \ref{propmain2ea} and \ref{propmain2eb}, $\widetilde M \coloneqq \lim_{s \to -\alpha}(s+\alpha)^{\beta+1} M_T(s)$, which is positive.
\end{proposition}
 \begin{proof} 
The uniqueness of $\alpha$ follows exactly as in Proposition \ref{propmain}: because $r_D(A(0))=0<\lambda_{\min}$ and $r_D(A(s))$ is convex, the equation $r_D(A(-\alpha))=\lambda_{\min}$ can have at most one solution with $\alpha>0$.

\smallskip\noindent\textbf{Proof of \ref{propmain2ea}.}
Conditional on $\lambda=\lambda_j$ (i.e., $L=j$), the trade time has law given by $T_j \stackrel{d}{=} T_{\lambda_j}+\sum_{k=1}^d \Erl(\mu_k,r_k)$. Whenever $M_T(s)$ exists, we find that
\begin{align}
M_T(s)
&= E\left[e^{sX_T}\right]
 = \sum_{j=1}^\tau q_j E\left[e^{sX_T}\mid L=j\right] \notag\\
&= w_0'\left(\sum_{j=1}^\tau q_j E\left[e^{A(s)T}\mid L=j\right]\right)1_N \notag\\
&= w_0'\left(\sum_{j\in\mathcal J^*} q_j E\left[e^{A(s)T_j}\right]\right)1_N
+ w_0'\left(\sum_{j\notin\mathcal J^*} q_j E\left[e^{A(s) T_j}\right]\right)1_N.
\label{eqinteg1e1}
\end{align}

Let $\eta := r_D(A(-\alpha)) = \lambda_{\min}>0$. By Assumption \ref{asshol}, $A(-z)$ is holomorphic on an open connected neighborhood of $[0,\alpha]$. Define
\[
F_{\eta}(z) \coloneqq A(-z)-\eta I, \qquad z \in [0,\alpha].
\]
For real $z$, the matrix $A(-z)$ is irreducible Metzler because its off-diagonal sign pattern is inherited from the irreducible generator $G$; therefore $F_{\eta}(z)=A(-z)-\eta I$ is also irreducible Metzler. Moreover,
\[
r_D(F_{\eta}(z)) = r_D(A(-z)) - \eta.
\]
Since $z \mapsto r_D(A(-z))$ is convex, so is $z \mapsto r_D(F_{\eta}(z))$ (see the proof of Proposition \ref{propmain}). Furthermore,
\[
r_D(F_{\eta}(0)) = -\eta < 0, \qquad r_D(F_{\eta}(\alpha)) = 0.
\]
Hence Lemma \ref{lemconvsimple} applies, showing that $F_{\eta}(z)^{-1}$ has a simple pole at $z=\alpha$. Equivalently, $(\eta I-A(s))^{-1}$ has a simple pole at $s=-\alpha$. If $x_{\eta}$ and $y_{\eta}$ denote right and left eigenvectors of $A(-\alpha)$ associated with the eigenvalue $\eta$, then
\begin{align} \label{residform}
R_{\eta} \coloneqq \lim_{s\to -\alpha}(s+\alpha)(\eta I-A(s))^{-1}
= x_{\eta}\left(-\,y_{\eta}' A^{(1)}(-\alpha)x_{\eta}\right)^{-1}y_{\eta}'
= \xi_{\eta} x_{\eta} y_{\eta}',
\end{align}
where $\xi_{\eta} = \left(-\,y_{\eta}' A^{(1)}(-\alpha)x_{\eta}\right)^{-1} > 0$.

\smallskip\noindent\textbf{Case \ref{propmain2enum1}: $\lambda_{\min} = \lambda_1 < \mu_1$.}
In this case, $1 \notin \mathcal J^*$ and $\lambda_1 < \mu_k$ for every $k$. Note that the first term in \eqref{eqinteg1e1} (over $j\in\mathcal J^*$) can be written as a sum of elements of the form $\tilde c_{k,\ell,j}(\mu_k I -A(s))^{-\ell}$ for $k,\ell$ in the appropriate ranges (see Lemma \ref{erlem}). Since $\lambda_1<\mu_k$ for all $k$, each $(\mu_k I -A(s))^{-\ell}$ is holomorphic at $s=-\alpha$, and hence the first term in \eqref{eqinteg1e1} is holomorphic at $s=-\alpha$. We deduce from \eqref{eqerlang001} and Lemma \ref{erlem} that the second term in \eqref{eqinteg1e1} (over $j\notin\mathcal J^*$) can be written as
\begin{align} \label{2term}
&\text{the 2nd term in \eqref{eqinteg1e1}} \notag \\
&= w_0' \sum_{j \notin \mathcal J^*} q_j \left( c_{1,1,j} (\lambda_j I - A(s))^{-1} + \sum_{k=2}^{d+1} \sum_{\ell=1}^{r_{k-1}} c_{k,\ell,j} (\mu_{k-1} I - A(s))^{-\ell} \right) 1_N,
\end{align}
where $c_{k,\ell,j}$ are the Erlang coefficients.

By \eqref{residform} with $\eta=\lambda_1$, the term $(\lambda_1 I-A(s))^{-1}$ has a simple pole at $s=-\alpha$ with residue $R_{\lambda_1}=\xi_{\lambda_1}x_{\lambda_1}y_{\lambda_1}'$. Moreover, since $\lambda_1 < \mu_k$ for every $k$ and the shape parameter associated with $\lambda_1$ is one, the explicit coefficient formula for generalized Erlang densities (Section \ref{sec_erlang}) gives
\[
c_{1,1,1}>0,
\]
so the leading coefficient of the principal part of \eqref{2term} does not vanish. Therefore \eqref{2term} can be written as $w_0'q_1c_{1,1,1}(\lambda_1 I-A(s))^{-1}1_N + H(s)$, where $H(s)$ satisfies $(s+\alpha)H(s)\to 0$ as $s\to -\alpha$. Consequently,
\begin{align}
\widetilde{M} \coloneqq \lim_{s \to -\alpha}(s+\alpha)M_T(s)
= q_1 c_{1,1,1} (w_0' R_{\lambda_1} 1_N)
= \xi_{\lambda_1} q_1 c_{1,1,1}(w_0' x_{\lambda_1}) (y_{\lambda_1}' 1_N) > 0, \label{resideq1122}
\end{align}
so $M_T(s)$ has a simple pole at $s=-\alpha$. Corollary \ref{cortau} with $\ell=1$ then yields constants $0<C_1\le C_2<\infty$ such that, for all sufficiently large $y$,
\[
\widetilde M C_1 \le y^\alpha \PP(P_T>y) \le \widetilde M C_2,
\]
which establishes \eqref{propmain2eeq3} with $\beta=0$ and proves case \ref{propmain2enum1}.

\smallskip\noindent\textbf{Case \ref{propmain2enum2}: $\lambda_{\min} = \mu_1 < \lambda_1$.}
In this case, $\mu_1\notin\{\lambda_1,\ldots,\lambda_\tau\}$, so $\lambda_j\neq\mu_1$ for every $j$. For every type $j$, the completion-stage block with rate $\mu_1$ appears with shape parameter $r_1$. Terms involving $(\lambda_j I-A(s))^{-1}$ are holomorphic at $s=-\alpha$ since $\lambda_j>\mu_1$, and terms involving $(\mu_k I-A(s))^{-\ell}$ with $k\ge 2$ are holomorphic since $\mu_k>\mu_1$. Thus the highest-order singular terms are exactly those involving $(\mu_1 I-A(s))^{-r_1}$. Combining this observation with the densities \eqref{eqerlang001}--\eqref{eqerlang002} and Lemma \ref{erlem},
\begin{align}
M_T(s) = w_0' \left(\sum_{j \notin \mathcal J^*} q_j c_{2,r_1,j} (\mu_1 I - A(s))^{-r_1}\right) 1_N + w_0' \left(\sum_{j \in \mathcal J^*} q_j \tilde c_{1,r_1,j} (\mu_1 I - A(s))^{-r_1}\right) 1_N + H(s),
\end{align}
for some $H(s)$ with $\lim_{s\to -\alpha}(s+\alpha)^{r_1}H(s)= 0$. Here the coefficient indexing follows \eqref{eqerlang001} for $j\notin\mathcal J^*$ and \eqref{eqerlang002} for $j\in\mathcal J^*$.

Since $\mu_1<\lambda_j$ for every $j$ and $\mu_1<\mu_k$ for $k\ge 2$, and the shape parameter associated with $\mu_1$ is $r_1$, the explicit coefficient formula in Section \ref{sec_erlang} gives
\[
c_{2,r_1,j} > 0 \quad \text{for } j\notin\mathcal J^*, \qquad \tilde c_{1,r_1,j} > 0 \quad \text{for } j\in\mathcal J^*,
\]
so the coefficient of the highest-order singular term is strictly positive. Setting $a_j = c_{2,r_1,j}$ for $j\notin\mathcal J^*$ and $a_j = \tilde c_{1,r_1,j}$ for $j\in\mathcal J^*$, we obtain
\begin{align}
\widetilde{M} &\coloneqq \lim_{s \to -\alpha}(s+\alpha)^{r_1}M_T(s) \notag \\
&= w_0'\left(\sum_{j=1}^{\tau} q_j a_j R_{\mu_1}^{r_1}\right) 1_N \notag \\
&=\xi_{\mu_1}^{r_1} \left(\sum_{j=1}^{\tau} q_j a_j\right) (w_0'x_{\mu_1}) (y_{\mu_1}'x_{\mu_1})^{r_1-1}(y_{\mu_1}'1_N) > 0, \label{resideq12}
\end{align}
where positivity follows from $\xi_{\mu_1}>0$, $a_j>0$, $q_j>0$, and the Perron--Frobenius positivity of $x_{\mu_1}, y_{\mu_1}$. Hence $M_T(s)$ has a pole of order exactly $r_1$ at $s=-\alpha$, and Corollary \ref{cortau} with $\ell=r_1$ yields constants $0<C_1\le C_2<\infty$ such that
\[
\widetilde M C_1 \le (\log y)^{-(r_1-1)} y^\alpha \PP(P_T>y) \le \widetilde M C_2 \label{eqtaue4444}
\]
for all sufficiently large $y$, which proves case \ref{propmain2enum2} with $\beta = r_1-1$.

\smallskip\noindent\textbf{Case \ref{propmain2enum3}: $\lambda_{\min} = \lambda_1 = \mu_1$.}
In this case, $1 \in \mathcal J^*$, $u(1) = 1$, and the multiplicity of $\mu_{u(1)}=\mu_1$ in the rate sequence of $T_1$ is $r_1+1$. Since $\mu_1 < \mu_k$ for $k\ge 2$ and $\mu_1 < \lambda_j$ for $j\ge 2$, the resolvents $(\lambda_j I - A(s))^{-\ell}$ and $(\mu_k I - A(s))^{-\ell}$ (for the relevant ranges of $\ell$) are well defined and holomorphic at $s = -\alpha$. By \eqref{residform} with $\eta=\mu_1$, $(\mu_1 I - A(s))^{-1}$ has a simple pole at $s=-\alpha$ with residue $R_{\mu_1}=\xi_{\mu_1}x_{\mu_1}y_{\mu_1}'$. From \eqref{eqerlang002} and Lemma \ref{erlem}, the contribution of $j\in\mathcal J^*$ with $j\neq 1$ to \eqref{eqinteg1e1} involves only $(\mu_1 I - A(s))^{-\ell}$ for $\ell \le r_1$, while the contribution of $j\notin\mathcal J^*$ has a pole of order at most $r_1$ at $s=-\alpha$. The unique term contributing a pole of order $r_1+1$ comes from $j=1$. Therefore
\begin{align} \label{2term2}
M_T(s) = w_0' q_1 \tilde c_{1,r_1+1,1} (\mu_1 I - A(s))^{-r_1-1} 1_N + H(s),
\end{align}
where $H(s)$ satisfies $(s+\alpha)^{r_1+1}H(s) \to 0$ as $s\to -\alpha$.

Since the shape parameter associated with $\mu_1(=\lambda_1)$ is $r_1+1$ and $\mu_1<\mu_k$ for $k\ge 2$, the explicit coefficient formula in Section \ref{sec_erlang} gives $\tilde c_{1,r_1+1,1}>0$. Therefore
\begin{align}
\widetilde{M} &\coloneqq \lim_{s \to -\alpha}(s+\alpha)^{r_1+1}M_T(s) \notag \\
&= w_0'\left(q_1 \tilde c_{1, r_1+1,1} R_{\mu_1}^{r_1+1}\right) 1_N \notag \\
&=\xi_{\mu_1}^{r_1+1} q_1 \tilde c_{1, r_1+1,1} (w_0'x_{\mu_1}) (y_{\mu_1}'x_{\mu_1})^{r_1}(y_{\mu_1}'1_N) > 0, \label{resideq122}
\end{align}
so $M_T(s)$ has a pole of order $r_1+1$ at $s=-\alpha$. Corollary \ref{cortau} with $\ell=r_1+1$ then yields
\[
\widetilde M C_1 \le (\log y)^{-r_1} y^\alpha \PP(P_T>y) \le \widetilde M C_2 \label{eqtaue44444}
\]
for all sufficiently large $y$, which proves case \ref{propmain2enum3} with $\beta = r_1$.

\smallskip\noindent\textbf{Proof of \ref{propmain2eb}.}
The argument is nearly identical to that of Proposition \ref{propmain}-\ref{propmainb} and is therefore omitted.
\end{proof}

The proof of Proposition \ref{propmain2e} also produces an explicit formula for $\widetilde M$ in each case, given respectively by \eqref{resideq1122}, \eqref{resideq12}, and \eqref{resideq122}. The three cases admit a unified description in terms of the order of the leading pole of $M_T(s)$ at $s=-\alpha$:
\begin{itemize}
\item In case \ref{propmain2enum1}, the pole is simple, and $P_T$ has a Pareto-type tail with no logarithmic correction.
\item In case \ref{propmain2enum2}, the pole has order $r_1$. When $r_1=1$, $P_T$ again has a Pareto-type tail with no logarithmic correction; when $r_1\ge 2$, the tail acquires a logarithmic correction of order $r_1-1$.
\item In case \ref{propmain2enum3}, the pole has order $r_1+1\ge 2$, and the tail always carries a logarithmic correction of order $r_1$.
\end{itemize}
Cases \ref{propmain2enum2} with $r_1\ge 2$ and \ref{propmain2enum3} are therefore qualitatively similar: in both, $y^\alpha \PP(P_T>y)$ diverges logarithmically, so $P_T$ has neither a Pareto-type nor a Paretian tail. What singles out case \ref{propmain2enum3} is that the logarithmic correction is forced even when $r_1=1$, because the coincidence $\lambda_1=\mu_1$ raises the multiplicity of the smallest rate by one. In all three cases, the tail exponent $\alpha$ is determined by $r_D(A(-\alpha)) = \lambda_{\min}$, and the slowest component of the trade-time distribution---the smallest type-specific arrival intensity or the smallest completion-stage rate, whichever is smaller---is the only object that enters the exponent. Faster components affect the asymptotic tail only through the scale constant $\widetilde M$.

We close this section with two remarks. The first establishes sharp Paretian tails under additional conditions, and the second develops analogous lower-tail results for both timing models.

\begin{remark}[Price rigidity and exact Paretian tails]
For non-financial goods, it may be realistic to assume that the underlying price process does not evolve continuously, i.e., $\mu_j=\sigma_j^2=0$ in the L\'evy exponent of $(X_t,t\ge 0)$. Under this price rigidity, Propositions \ref{propmain} and \ref{propmain2e} guarantee only Pareto-type tails in general. To obtain an exact Paretian tail, Corollary \ref{cortau} requires not only that the relevant singularity lie on the axis of convergence, but also that the corresponding pole be simple.

For the IIM, let $\alpha>0$ satisfy $r_D(A(-\alpha))=-\log(1-p_1)$. A sufficient condition for $P_T$ to have a Paretian tail is
\begin{align}\label{eqsufficient}
\tau(A(-\alpha+\beta\ii)) < \tau(A(-\alpha)) \qquad\text{for all }\beta\in\mathbb R\setminus\{0\}.
\end{align}
Indeed, by the spectral mapping theorem,
\[
\rho\!\left((1-p_1)e^{A(-\alpha+\beta\ii)}\right) = (1-p_1)e^{\tau(A(-\alpha+\beta\ii))} < (1-p_1)e^{\tau(A(-\alpha))} = 1,
\]
so $I-(1-p_1)e^{A(-\alpha+\beta\ii)}$ is invertible for every $\beta\neq 0$. Hence $s=-\alpha$ is the unique singularity of $M_T(s)$ on its axis of convergence, and Proposition \ref{propmain}-\ref{propmainb} follows.

For the ITM, let $\alpha>0$ satisfy $r_D(A(-\alpha))=\lambda_{\min}$. Condition \eqref{eqsufficient} is again sufficient for uniqueness of the singularity, but this yields an exact Paretian tail only in the simple-pole cases of Proposition \ref{propmain2e}---namely case \ref{propmain2enum1}, and case \ref{propmain2enum2} with $r_1=1$. In case \ref{propmain2enum2} with $r_1\ge 2$, and in case \ref{propmain2enum3}, the leading pole has order greater than one, so the logarithmic correction in Proposition \ref{propmain2e} remains. As in \citet{bst2018}, a nonlattice condition is a convenient sufficient condition for \eqref{eqsufficient}; in particular, a nondegenerate diffusion component in every regime suffices, though it is not necessary.

To illustrate that exact Paretian tails can arise even under price rigidity, consider the scalar case $N=1$ in the ITM, with $X$ a compound Poisson process with normally distributed jumps. Its L\'evy exponent is
\[
\psi(z) = \kappa\left(e^{\mu z+\frac12\sigma^2 z^2}-1\right), \qquad \kappa>0,\ \sigma^2>0.
\]
Let $-\alpha$ satisfy $\psi(-\alpha)=\lambda_{\min}$. Then, for every $\beta\neq 0$,
\begin{align*}
\RRe \psi(-\alpha+\beta\ii)
&= \kappa\left(e^{-\mu\alpha+\frac12\sigma^2(\alpha^2-\beta^2)}\cos(\mu\beta-\sigma^2\alpha\beta)-1\right) \\
&\le \kappa\left(e^{-\mu\alpha+\frac12\sigma^2(\alpha^2-\beta^2)}-1\right)
< \kappa\left(e^{-\mu\alpha+\frac12\sigma^2\alpha^2}-1\right) = \psi(-\alpha).
\end{align*}
Hence \eqref{eqsufficient} holds. Whenever the corresponding leading pole is simple (e.g., in case \ref{propmain2enum1}, or in case \ref{propmain2enum2} with $r_1=1$), $P_T$ has a Paretian tail even though the latent price process has no continuous evolution between jumps. 
\end{remark}

\begin{remark}[Lower tails]\label{remlowertail}
Since $P_T = e^{X_T}>0$, the natural lower-tail analogue concerns small realized prices:
\[
\PP(P_T<y^{-1}) = \PP(X_T<-\log y), \qquad y>1.
\]
Let $\widetilde X_t := -X_t$ and $\widetilde P_T := e^{\widetilde X_T} = P_T^{-1}$. Then $(\widetilde X, J)$ is again a Markov-modulated L\'evy process, with matrix Laplace exponent
\[
\widetilde A(s) = A(-s).
\]
Therefore the upper-tail results in Propositions \ref{propmain} and \ref{propmain2e} apply directly to $\widetilde P_T$.

For the IIM, if there exists $\gamma>0$ such that $r_D(A(\gamma)) = -\log(1-p_1)$, then $\gamma$ is unique and there exist constants $0<C_1\le C_2<\infty$ such that
\[
C_1 \le y^\gamma \PP(P_T<y^{-1}) \le C_2
\]
for all sufficiently large $y$, or equivalently
\[
C_1 \le e^{\gamma x}\PP(X_T<-x) \le C_2
\]
for all sufficiently large $x$. Under the corresponding uniqueness-of-singularity condition (see Corollary \ref{cortau} and Remark \ref{remlowtail}), the exact lower-tail analogue of Proposition \ref{propmain}-\ref{propmainb} follows.

For the ITM, if there exists $\gamma>0$ such that $r_D(A(\gamma)) = \lambda_{\min}$, then $\gamma$ is unique and there exist a nonnegative integer $k$ and constants $0<C_1\le C_2<\infty$ such that
\[
C_1 \le (\log y)^{-k} y^\gamma \PP(P_T<y^{-1}) \le C_2
\]
for all sufficiently large $y$, or equivalently
\[
C_1 \le x^{-k} e^{\gamma x}\PP(X_T<-x) \le C_2
\]
for all sufficiently large $x$. The value of $k$ is determined by the same case distinction as in Proposition \ref{propmain2e}, since the trade-time distribution is unchanged.
\end{remark}

\section{Conclusion}
This paper studies how random trade timing can generate heavy-tailed realized prices from light-tailed latent price dynamics. The realized price is modeled as $P_T=e^{X_T}$, where $X$ is a Markov-modulated L\'evy process and $T$ is the time of the next trade. In the intertrade-incidence model with geometric waiting time, the upper tail of $P_T$ is Pareto-type when the dominant real eigenvalue of the matrix Laplace exponent reaches the threshold determined by the smallest trade probability. A negative-binomial extension of the same model shows how higher-order discrete waiting mechanisms generate logarithmic corrections: when the realized price is observed at the $n$-th successful trading opportunity, the Pareto exponent is unchanged, but the tail acquires a logarithmic correction of order $n-1$. In the intertrade-time model, the same pole-based mechanism extends to generalized Erlang waiting times, with logarithmic corrections arising in boundary cases where the smallest rate appears with higher multiplicity.

The central economic implication is that the slowest component of the trade-time distribution determines the tail exponent. In the incidence model, this component is the least frequent trading type. In the intertrade-time model, it is the smallest rate among arrival and completion stages. Faster components affect the scale of the tail but not its decay exponent. Thus the lower tail of the trading-frequency distribution---or, more generally, the bottleneck in trade timing---is the key object for understanding the upper tail of realized prices. Analogous lower-tail results follow by applying the same arguments to the dual process $-X$, yielding corresponding asymptotics for small realized prices.

The model is deliberately tractable, and several extensions remain open. One natural next step is to estimate or calibrate the model using transaction-level data. Another is to study multi-trade returns rather than a single realized price. A further extension would allow feedback from trading activity to the latent price dynamics. These directions would help connect the theoretical mechanism developed here to empirical asset-pricing applications.

    \bibliographystyle{apalike}

\appendix
\section{Properties of Metzler matrices} \label{appmetzler}
In this appendix, we collect properties of Metzler matrices that are used repeatedly in this paper. Nonnegative matrices are special cases of Metzler matrices, so we first record the relevant result on nonnegative matrices. As described in Section \ref{smain}, $\rho(\mathbb{A})$, $\tau(\mathbb{A})$, and $r_D(\mathbb{A})$ denote the spectral radius, the spectral abscissa, and the dominant real eigenvalue of a square matrix $\mathbb{A}$, respectively.

\begin{lemma}[Perron--Frobenius] \label{pftheorem}
Let $\mathbb{A}$ be an irreducible nonnegative $N\times N$ matrix. Then $\rho(\mathbb{A})$ is strictly positive and a simple eigenvalue of $\mathbb{A}$, and the corresponding left and right eigenvectors $x=(x_1,\ldots,x_N)'$ and $y=(y_1,\ldots,y_N)'$ satisfy $x_j>0$ and $y_j>0$ for all $j=1,\ldots,N$.
\end{lemma}

The Perron--Frobenius theorem yields the following properties of Metzler matrices.
\begin{lemma} \label{metzlerthm}
Let $\mathbb{A}$ be a Metzler matrix. Then
\begin{enumerate}[label=(\roman*), ref=(\roman*), series=metzler]
    \item \label{metzlerthma} $r_D(\mathbb{A})$ exists.
    \item \label{metzlerthmb} $r_D(\mathbb{A}) \le \tau(\mathbb{A}+\mathbb{B})$ for any nonnegative matrix $\mathbb{B}$.
    \item \label{metzlerthmc} $(sI-\mathbb{A})^{-1}$ exists and is nonnegative if and only if $s > r_D(\mathbb{A})$.
\end{enumerate}
If, in addition, $\mathbb{A}$ is irreducible, then
\begin{enumerate}[label=(\roman*), ref=(\roman*), resume=metzler]
    \item \label{metzlerthmd} the left and right eigenvectors associated with $r_D(\mathbb{A})$ may be chosen with strictly positive components.
\end{enumerate}
\end{lemma}

\begin{proof}
Parts \ref{metzlerthma} and \ref{metzlerthmc} follow from Proposition 1-(i) and Proposition 1-(iv) of \citet{son1996}, respectively. Given the existence of $r_D(\mathbb{A})$ from \ref{metzlerthma}, the inequality in \ref{metzlerthmb} follows from Lemma 2 of \citet{son1996}.

To show \ref{metzlerthmd}, note that since $\mathbb{A}$ is Metzler, we can choose a positive real number $\upsilon$ such that $\upsilon I + \mathbb{A}$ is nonnegative. Because $\mathbb{A}$ is irreducible, so is $\upsilon I + \mathbb{A}$. Lemma \ref{pftheorem} then implies that $\rho(\upsilon I + \mathbb{A})$ is a simple real eigenvalue of $\upsilon I + \mathbb{A}$ with strictly positive left and right eigenvectors. Since
\[
\Sigma(\upsilon I + \mathbb{A}) = \{\upsilon + r_i : r_i \in \Sigma(\mathbb{A})\},
\]
the value $\rho(\upsilon I + \mathbb{A}) - \upsilon$ is a real eigenvalue of $\mathbb{A}$, and a direct check shows that it satisfies all the properties required of the dominant real eigenvalue. The corresponding eigenvectors of $\mathbb{A}$ coincide with those of $\upsilon I + \mathbb{A}$ and are therefore strictly positive.
\end{proof}
  
\section{Tauberian theorem}\label{apptauberian}

This appendix collects the Tauberian results used to characterize the tail behavior in the main text. Nakagawa's
Tauberian theorem is stated for nonnegative random variables, whereas the log-price
 $X_T$ in the body of the paper is real-valued. We therefore first reduce the
upper-tail problem to the positive part of a real-valued random variable, and then state
Nakagawa's theorem in the notation used here.
\begin{lemma}[Reduction to the positive part] \label{lemxplus}
Let $X$ be a real-valued random variable, and define $X^+ \coloneqq \max\{X,0\}$. Let $\varphi_X(s) \coloneqq E[e^{-sX}]$ and $\varphi_+(s) \coloneqq E[e^{-sX^+}]$ wherever finite. Suppose that $\varphi_X$ satisfies the following conditions:
\begin{enumerate}[label=(\roman*), ref=(\roman*)]
    \item \label{lemxpluda} The abscissa of convergence of $\varphi_X$ is $-\alpha$, with $0 < \alpha < \infty$.
    \item \label{lemxpludb} $s = -\alpha$ is a pole of order $\ell$, and $A_\ell$ denotes the coefficient of $(s+\alpha)^{-\ell}$ in the Laurent expansion of $\varphi_X$ at $s=-\alpha$.
    \item \label{lemxpludc} $\varphi_X$ is holomorphic on an open neighborhood of the segment
    \[
    \{s=-\alpha+\ii\tau : -2\alpha\delta<\tau<2\alpha\delta\}\setminus\{-\alpha\}
    \]
    for some $\delta>0$.
\end{enumerate}
Then $\varphi_+$ also satisfies \ref{lemxpluda}--\ref{lemxpludc} with the same constants $\alpha$, $\ell$, $A_\ell$, and $\delta$. If, in addition, $-\alpha$ is the unique singularity of $\varphi_X$ on the axis $\RRe(s)=-\alpha$, then the same holds for $\varphi_+$.
\end{lemma}

\begin{proof}
Decompose
\[
\varphi_+(s) = E[e^{-sX}\mathbbm{1}(X>0)] + \PP(X\le 0) = \varphi_X(s) + H(s),
\]
where
\[
H(s) \coloneqq \PP(X\le 0) - E[e^{-sX}\mathbbm{1}(X\le 0)].
\]
We first show that $H$ is holomorphic on the half-plane $\{\RRe(s)<0\}$. Let $\nu$ be the finite measure on $[0,\infty)$ defined by
\[
\nu(\mathcal B) \coloneqq \PP(-X\in \mathcal B,\ X\le 0).
\]
Then for $s$ with $\RRe(s)<0$,
\begin{equation}\label{eqmethod}
E[e^{-sX}\mathbbm{1}(X\le 0)] = \int_{[0,\infty)} e^{su}\,\nu(du).
\end{equation}
Fix $s_0$ with $\RRe(s_0)<0$, and choose $\eta>0$ with $\RRe(s_0)+\eta<0$. For $|s-s_0|<\eta$,
\[
e^{su} = e^{s_0u}\sum_{n=0}^\infty \frac{(s-s_0)^n u^n}{n!},
\]
and
\[
\sum_{n=0}^\infty \left|e^{s_0u}\frac{(s-s_0)^n u^n}{n!}\right| \le e^{(\RRe(s_0)+\eta)u} \le 1.
\]
Since $\nu$ is a finite measure, the dominated convergence theorem yields the following
\[
\int_{[0,\infty)} e^{su}\,\nu(du) = \sum_{n=0}^\infty \frac{(s-s_0)^n}{n!} \int_{[0,\infty)} u^n e^{s_0u}\,\nu(du)
\]
on $|s-s_0|<\eta$. Hence $H$ is holomorphic on $\{\RRe(s)<0\}$.

Since $-\alpha<0$, $H$ is holomorphic in a neighborhood of $s=-\alpha$. Therefore $\varphi_+$ and $\varphi_X$ have the same principal part in their Laurent expansions at $s=-\alpha$. In particular, if
\[
\varphi_X(s) = \sum_{m=-\ell}^{\infty} a_m (s+\alpha)^m
\]
near $s=-\alpha$, then
\[
\varphi_+(s) = \sum_{m=-\ell}^{-1} a_m (s+\alpha)^m + \sum_{m=0}^{\infty} \tilde a_m (s+\alpha)^m
\]
for suitable coefficients $\tilde a_m$. Thus $s=-\alpha$ is a pole of $\varphi_+$ of order $\ell$, and the coefficient of $(s+\alpha)^{-\ell}$ is again $A_\ell = a_{-\ell}$.

Because $H$ is holomorphic on the entire half-plane $\{\RRe(s)<0\}$, the holomorphy of $\varphi_X$ on an open neighborhood of
\[
\{s=-\alpha+\ii\tau : -2\alpha\delta<\tau<2\alpha\delta\}\setminus\{-\alpha\}
\]
is equivalent to the same property for $\varphi_+$. Likewise, $-\alpha$ is the unique singularity of $\varphi_X$ on the axis $\RRe(s)=-\alpha$ if and only if it is the unique singularity of $\varphi_+$ there.

Finally, for any real $s>-\alpha$,
\[
\varphi_+(s) = \PP(X\le 0) + E[e^{-sX}\mathbbm{1}(X>0)] \le \PP(X\le 0) + E[e^{-sX}] < \infty,
\]
so the abscissa of convergence of $\varphi_+$ is at most $-\alpha$. Since $s=-\alpha$ remains a singularity of $\varphi_+$, its abscissa of convergence is exactly $-\alpha$.
\end{proof}

\begin{remark}[Lower-tail reduction]  \label{remlowtail}
Lemma \ref{lemxplus} immediately yields the corresponding reduction for lower tails
after replacing \(X\) by \(-X\). Indeed,
\[
\PP(X<-x)=\PP((-X)^+>x), \qquad x>0.
\]
Hence any upper-tail Tauberian statement for a real-valued random variable applies
to the lower tail after passing to its negative.
\end{remark}

We next state a version of the Tauberian theorem of \citet{nakagawa2007} for a nonnegative random variable \(Y\).  The corresponding Laplace transform is denoted by $\varphi_Y(s)$. Suppose that the
abscissa of convergence of $\varphi_Y$ is $-\alpha$ with $0<\alpha<\infty$, and that
$s=-\alpha$ is a pole of order $\ell$.

To state the theorem, let
\[
K=
\begin{cases}
\ell, & \text{if $\ell$ is odd},\\
\ell+1, & \text{if $\ell$ is even}.
\end{cases}
\]
We define for $w \in \mathbb{R}$
\[
R(w)=\frac{w^K}{1-e^{-w}}, \qquad r(w)=w^K.\]
For $-\infty<t<\infty$, we let
\begin{align}
Q_{\omega}(t)
&\coloneqq \sum_{n=0}^\infty e^{-n\omega}
\sum_{k=1}^{K+1} \frac{(-1)^k r^{(k-1)}(\omega)}{(t-n)^k}
+ \sum_{k=1}^{K} \frac{(-1)^{k-1} R^{(k-1)}(\omega)}{t^k}, \\
M_{\omega}(t)
&\coloneqq \frac{1}{K!}\left(\frac{\sin \pi t}{\pi}\right)^{K+1} Q_{\omega}(t), \\
m_{\omega}(t)
&\coloneqq M_{\omega}(t) - \frac{1}{1-e^{-\omega}}
\left(\frac{\sin \pi t}{\pi t}\right)^{K+1}.
\end{align}
For $\delta=2\pi/\omega$ and $t$, set
\[
M_{\delta,\alpha}(t)\coloneqq M_{\omega}\!\left(\frac{\alpha t}{\omega}\right),
\qquad
m_{\delta,\alpha}(t)\coloneqq m_{\omega}\!\left(\frac{\alpha t}{\omega}\right).
\]

\begin{thm}[Nakagawa \citeyearpar{nakagawa2007}, Theorem 5$^\ast$]\label{thmtau}
Let $s=-\alpha$ be a pole of $\varphi_Y(s)$ of order $\ell$, and suppose that
$\varphi_Y$ is holomorphic on an open neighborhood of the segment
\[
\{\,s=-\alpha+\ii\tau:\,-2\alpha\delta<\tau<2\alpha\delta\,\}\setminus\{-\alpha\}
\]
for some $\delta>0$. Let $A_\ell$ be the coefficient of $(s+\alpha)^{-\ell}$ in
the Laurent expansion of $\varphi_Y(s)$ at $s=-\alpha$. Then
\begin{align}
A_\ell \int_{-\infty}^{\infty} m_{\delta,\alpha}(t)\,dt
&\le \liminf_{x\to\infty} x^{-\ell+1} e^{\alpha x} \PP(Y>x) \\
&\le \limsup_{x\to\infty} x^{-\ell+1} e^{\alpha x} \PP(Y>x) \\
&\le A_\ell \int_{-\infty}^{\infty} M_{\delta,\alpha}(t)\,dt.
\end{align}
\end{thm}
From Theorem \ref{thmtau}, we obtain the following extension to real-valued
random variables:
\begin{cor}[Extension to real-valued random variables]\label{cortau}
Let $X$ be a real-valued random variable whose Laplace transform $\varphi_X$ satisfy the conditions \ref{lemxpluda}--\ref{lemxpludc} of Lemma \ref{lemxplus}. Then there exist finite constants $C_1$ and $C_2$, depending only on
$\alpha$, $\ell$, and $\delta$, such that
\begin{align}
A_\ell C_1
&\le \liminf_{x\to\infty} x^{-\ell+1} e^{\alpha x} \PP(X>x) \notag\\
&\le \limsup_{x\to\infty} x^{-\ell+1} e^{\alpha x} \PP(X>x)
\le A_\ell C_2. \label{highereq}
\end{align}
If, in addition, $\ell=1$ and $-\alpha$ is the unique singularity of $\varphi_X$ on the
axis $\RRe (s)=-\alpha$, then
\begin{align}
\lim_{x\to\infty} e^{\alpha x}\PP(X>x) = A_1/\alpha. \label{simpoleeq}
\end{align}
\end{cor}

\begin{proof}
Let $Y:=X^+$. By Lemma \ref{lemxplus}, the Laplace transform of $Y$ has the same
abscissa of convergence, the same pole order, and the same leading Laurent coefficient
as $\varphi_X$. Moreover,
\[
\PP(Y>x)=\PP(X>x), \qquad x>0.
\]
Applying Theorem \ref{thmtau} to the nonnegative random variable $Y$, we obtain
\eqref{highereq} with
\[
C_1:=\int_{-\infty}^{\infty} m_{\delta,\alpha}(t)\,dt,
\qquad
C_2:=\int_{-\infty}^{\infty} M_{\delta,\alpha}(t)\,dt.
\]
If $\ell=1$ and $-\alpha$ is the unique singularity of $\varphi_X$ on the axis
$\RRe (s)=-\alpha$, then the same is true for the Laplace transform of $Y$ by
Lemma \ref{lemxplus}. Hence \eqref{simpoleeq} follows from
\citet[Theorem 5.2 and 5.4]{korevaar2004} or
\citet[Theorem 2.1 of the working paper version]{bearetoda2017}, applied to $Y$.
\end{proof}

   
\section{Simple poles of holomorphic matrix-valued functions}\label{appsimpole}

This appendix collects two lemmas characterizing simple poles of the inverse of a holomorphic matrix-valued function.

\begin{lemma}[Simple-pole criterion] \label{lemsimpole}
Let $\mathbb{A}(z)$ be an $N\times N$ complex matrix-valued holomorphic function on an open connected set $\Omega \subset \mathbb C$. Suppose that $\mathbb{A}(z)$ is invertible for some $z \in \Omega$, and that $\mathbb{A}(z_0)$ has rank $r<N$ for some $z_0 \in \Omega$. Then the following conditions are equivalent:
\begin{enumerate}[label=(\roman*), ref=(\roman*)]
    \item \label{lemsimpolea} $\mathbb{A}(z)^{-1}$ has a simple pole at $z=z_0$.
    \item \label{lemsimpoleb} There exist $N \times (N-r)$ matrices $x$ and $y$ of full column rank such that $\mathbb{A}(z_0)x = 0$, $y^\top \mathbb{A}(z_0) = 0$, and $y^\top \mathbb{A}^{(1)}(z_0)x$ is invertible.
\end{enumerate}
Under either condition, the residue of $\mathbb{A}(z)^{-1}$ at $z=z_0$ is
\[
x \left(y^\top \mathbb{A}^{(1)}(z_0) x\right)^{-1} y^\top.
\]
\end{lemma}
\begin{proof}
See Theorem 1 of \citet{bst2025corr}.
\end{proof}

\begin{lemma}[Convex spectral-abscissa criterion] \label{lemconvsimple}
Let $\mathbb{A}(z)$ be an $N\times N$ complex matrix-valued holomorphic function on an open connected set $\Omega \subset \mathbb C$ containing the interval $[0,\alpha]$. Suppose that
\begin{enumerate}[label=(\roman*), ref=(\roman*)]
    \item \label{lemconvsimplea} $\mathbb{A}(z)$ is irreducible Metzler for every real $z$ in some open interval containing $[0,\alpha]$;
    \item \label{lemconvsimpleb} $r_D(\mathbb{A}(z))$ is convex in real $z \in [0,\alpha]$;
    \item \label{lemconvsimplec} $r_D(\mathbb{A}(0))<0$ and $r_D(\mathbb{A}(\alpha))=0$.
\end{enumerate}
Then $\mathbb{A}(\alpha)$ has rank $N-1$, $\mathbb{A}(z)^{-1}$ has a simple pole at $z=\alpha$, and if $x$ and $y$ are right and left eigenvectors of $\mathbb{A}(\alpha)$ associated with the eigenvalue $0$, then the residue of $\mathbb{A}(z)^{-1}$ at $z=\alpha$ is
\[
x \left(y^\top \mathbb{A}^{(1)}(\alpha) x\right)^{-1} y^\top.
\]
\end{lemma}

\begin{proof}
Set $g(z) := r_D(\mathbb{A}(z))$. Because $\mathbb{A}(z)$ is irreducible Metzler for real $z\in[0,\alpha]$, its dominant real eigenvalue coincides with its spectral abscissa on that interval, so $g(z) = \tau(\mathbb{A}(z))$. By Lemma \ref{metzlerthm}-\ref{metzlerthmc}, $\mathbb{A}(0)$ is invertible since $g(0) < 0$. The convexity of $g$ on $[0,\alpha]$, together with the endpoint conditions $g(0) < 0$ and $g(\alpha) = 0$, implies that the left derivative at $\alpha$ is strictly positive; indeed,
\[
\partial_- g(\alpha) \ge \frac{g(\alpha) - g(0)}{\alpha} > 0.
\]
Hence the hypotheses of Lemma 1 in \citet{bst2025corr} are satisfied with $z_0 = \alpha$, so $\mathbb{A}(\alpha)$ has rank $N-1$ and $\mathbb{A}(z)^{-1}$ has a simple pole at $z=\alpha$. The residue formula then follows from Theorem 1 of \citet{bst2025corr}.
\end{proof}

\section{Useful results on generalized Erlang random variables} \label{sec_erlang}
Consider any generalized Erlang random variable of the form
\[
\mathbb{T}=\sum_{k=1}^D \Erl(a_k,b_k),
\]
with pairwise distinct rates $a_1,\ldots,a_D$. For background on generalized Erlang distributions and their structural properties, see \citet{dehon1982}. \citet{jasiulewicz2003} provide an explicit probability density function of $\mathbb{T}$, given by
\[
f_{\mathbb{T}}(t)=\sum_{k=1}^D \sum_{\ell=1}^{b_k} \frac{c_{k,\ell}}{(\ell-1)!}\, t^{\ell-1}e^{-a_k t}, \qquad t>0,
\]
where
\begin{equation}\label{coeffc}
c_{k,\ell}
= \left(\prod_{m=1}^D a_m^{b_m}\right)(-1)^{b_k-\ell}
\sum_{\substack{d_1,\ldots,d_D\ge 0 \\ d_k=0,\ \sum_{m=1}^D d_m=b_k-\ell}}
\prod_{\substack{m=1 \\ m\neq k}}^D \binom{b_m+d_m-1}{d_m} \frac{1}{(a_m-a_k)^{b_m+d_m}}.
\end{equation}
Building on this density expansion, the following lemma provides a useful expression for the matrix-valued expectation $E[e^{\mathbb{A}\mathbb{T}}]$.

\begin{lemma}\label{erlem}
Let $\mathbb T = \sum_{k=1}^D \Erl(a_k,b_k)$ be a generalized Erlang random variable with distinct rates $0<a_1<\cdots<a_D$, and let $\mathbb A$ be a square matrix with $\tau(\mathbb A)<a_1$. Then
\begin{equation}\label{erexpan1}
E\left[e^{\mathbb A\mathbb T}\right] = \sum_{k=1}^D\sum_{\ell=1}^{b_k} c_{k,\ell}(a_kI-\mathbb A)^{-\ell}.
\end{equation}
Moreover, if $\mathbb A(z)$ is a holomorphic matrix-valued function on an open connected set $\Omega\subset\mathbb C$, and if $a_kI-\mathbb A(z)$ is invertible for at least one $z\in\Omega$ for each $k=1,\ldots,D$, then
\[
\mathcal G(z) := \sum_{k=1}^D\sum_{\ell=1}^{b_k} c_{k,\ell}(a_kI-\mathbb A(z))^{-\ell}
\]
is meromorphic on $\Omega$, and its possible poles occur only at points $z\in\Omega$ where $a_kI-\mathbb A(z)$ is noninvertible for at least one $k$.
\end{lemma}

\begin{proof}
By the density expansion of the generalized Erlang distribution,
\begin{equation}\label{eqlan1}
E\left[e^{\mathbb A\mathbb T}\right]
= \sum_{k=1}^D\sum_{\ell=1}^{b_k} c_{k,\ell} \int_0^\infty \frac{t^{\ell-1}}{(\ell-1)!} e^{-a_kt}e^{\mathbb A t}\,dt
= \sum_{k=1}^D\sum_{\ell=1}^{b_k} c_{k,\ell} \int_0^\infty \frac{t^{\ell-1}}{(\ell-1)!} e^{-(a_kI-\mathbb A)t}\,dt,
\end{equation}
where the second equality uses $a_kI$ commutes with $\mathbb A$. Since $\tau(\mathbb A)<a_1\le a_k$, all eigenvalues of $a_kI-\mathbb A$ have strictly positive real parts, so the integral $\int_0^\infty e^{-(a_kI-\mathbb A)t}\,dt$ is well defined and equals
\[
\left[-(a_kI-\mathbb A)^{-1} e^{-(a_kI-\mathbb A)t}\right]_0^\infty = (a_kI-\mathbb A)^{-1}.
\]
For $\ell\ge 2$, integration by parts gives
\begin{align*}
\int_0^\infty \frac{t^{\ell-1}}{(\ell-1)!} e^{-(a_kI-\mathbb A)t}\,dt
&= \left[\frac{t^{\ell-1}}{(\ell-1)!}\left(-(a_kI-\mathbb A)^{-1} e^{-(a_kI-\mathbb A)t}\right)\right]_0^\infty \\
&\quad + (a_kI-\mathbb A)^{-1} \int_0^\infty \frac{t^{\ell-2}}{(\ell-2)!} e^{-(a_kI-\mathbb A)t}\,dt \\
&= (a_kI-\mathbb A)^{-1} \int_0^\infty \frac{t^{\ell-2}}{(\ell-2)!} e^{-(a_kI-\mathbb A)t}\,dt.
\end{align*}
Combining this recursion with the base case $\ell=1$, induction on $\ell$ yields
\begin{equation}\label{eqlan2}
\int_0^\infty \frac{t^{\ell-1}}{(\ell-1)!} e^{-(a_kI-\mathbb A)t}\,dt = (a_kI-\mathbb A)^{-\ell}, \qquad \ell\ge 1.
\end{equation}
Substituting \eqref{eqlan2} into \eqref{eqlan1} establishes \eqref{erexpan1}.

For the second claim, if $\mathbb A(z)$ is holomorphic on $\Omega$, then so is $a_kI-\mathbb A(z)$. Since $a_kI-\mathbb A(z)$ is invertible for at least one $z\in\Omega$, the holomorphic function $\det(a_kI-\mathbb A(z))$ is not identically zero on $\Omega$. Whenever this determinant is nonzero,
\[
(a_kI-\mathbb A(z))^{-1} = \frac{\operatorname{adj}(a_kI-\mathbb A(z))}{\det(a_kI-\mathbb A(z))},
\]
where $\operatorname{adj}$ denotes the adjugate matrix. Hence each entry of $(a_kI-\mathbb A(z))^{-1}$ is a ratio of holomorphic functions with denominator not identically zero, so $(a_kI-\mathbb A(z))^{-1}$ is meromorphic on $\Omega$, and so are its powers. Since $\mathcal G$ is a finite linear combination of such terms, $\mathcal G$ is meromorphic on $\Omega$. Its possible poles can occur only where $\det(a_kI-\mathbb A(z))=0$ for some $k$, equivalently where $a_kI-\mathbb A(z)$ is noninvertible for at least one $k$.
\end{proof}

\end{document}